\documentclass[11pt,a4paper,oneside,reqno]{amsart}

\usepackage[all]{xy}
\SelectTips{cm}{}  \calclayout
\usepackage{amsfonts,amssymb,amscd,amsmath,enumerate,verbatim,calc}

\usepackage{amscd,amssymb,amsopn,amsmath,amsthm,graphics,amsfonts,enumerate,verbatim,calc}
\usepackage[dvips]{graphicx}
\usepackage[colorlinks=true,linkcolor=blue,citecolor=blue]{hyperref}
\input xy
\xyoption{all}

\newtheorem{theorem}{Theorem}[section]
\newtheorem{corollary}[theorem]{Corollary}
\newtheorem{lemma}[theorem]{Lemma}

\theoremstyle{definition}
\newtheorem{definition}[theorem]{Definition}
\newtheorem{example}[theorem]{Example}

\newtheorem{remark}[theorem]{Remark}

\theoremstyle{plain}

\theoremstyle{definition}

\numberwithin{equation}{section}
\begin{document}

\title[On centrally extended Jordan derivations]{On centrally extended Jordan derivations and related maps in rings}
\author[ B. Bhushan, G. S. Sandhu, S. Ali and D. Kumar
]{ B. Bhushan, G. S. Sandhu, S. Ali and D. Kumar
}

\address{Department of Mathematics, Faculty of Physical Sciences, Punjabi University, Patiala-147002, India.}
\email{bharat\_rs18@pbi.ac.in}
\address{Department of Mathematics, Patel Memorial National College, Rajpura-140401, India.}
\email{gurninder\_rs@pbi.ac.in}
\address{Department of Mathematics, Aligarh Muslim University, Aligarh-202002, India.}
\email{shakir50@rediffmail.com}
\address{Department of Mathematics, Faculty of Physical Sciences, Punjabi University, Patiala-147002, India.}
\email{deep\_math1@yahoo.com}

\begin{abstract}
Let $R$ be a ring and $Z(R)$ be the center of $R.$ The aim of this paper is to define the notions of centrally extended Jordan derivations and centrally extended Jordan $\ast$-derivations, and to prove some results involving these mappings. Precisely, we prove that if a $2$-torsion free noncommutative prime ring $R$ admits a centrally extended Jordan derivation (resp. centrally extended Jordan $\ast$-derivation) $\delta:R\to R$ such that
\[
[\delta(x),x]\in Z(R)~~(resp.~~[\delta(x),x^{\ast}]\in Z(R))\text{~for~all~}x\in R,
\]
 where $'\ast'$ is an involution on $R,$ then $R$ is an order in a central simple algebra of dimension at most 4 over its center. 
\end{abstract}
\maketitle

\section{Introduction and notions}
Throughout this paper, $R$ denotes an associative ring with center $Z(R)$. The maximal right ring of quotients of $R$ is denoted by $Q_{mr}(R)$ and the center of $Q_{mr}(R)$ is called the extended centroid of $R$ and denoted by $C,$ more information about these object can be found in \cite{Beidar1996}. For any $x, y\in R$, the symbol $[x,y]$ (resp. $x \circ y$) denotes the commutator (resp. anti-commutator)  $xy-yx$ (resp. $xy+yx$). A ring $R$ is called \emph{prime}, if for any $a,b\in R$, $aRb=(0)$ implies either $a=0$ or $b=0,$ and if $aRa=(0)$ implies $a=0,$ then $R$ is called a \emph{semiprime} ring. For any $n\in\mathbb{Z}^{+},$ $R$ is called $n$-torsion free if $nx=0$ for all $x\in R,$ implies $x=0$. An anti-automorphism $'\ast'$ of a ring $R$ is called \emph{involution} if it is of period $2$. By a \emph{ring with involution}, we mean a ring equipped with an involution $'\ast',$ it is also called $\ast$-ring.
Let $H(R):=\{x\in R: x^{\ast}=x\}$ and $S(R):=\{x\in R: x^{\ast}=-x\};$ the elements of $H(R)$ are called \emph{symmetric} and the elements of $S(R)$ are called \emph{skew-symmetric}. Following Herstein \cite[Ch. 6]{Herstein76}, $\overline{H(R)}$ will denote the ring generated by the symmetric elements of $R$.
\par
An additive mapping $d:R\to R$ is called a \emph{derivation} if $d(xy)=d(x)y+xd(y)$ for all $x,y\in R.$ For a fixed element $a\in R,$ a mapping $x\mapsto [a,x]$ is called \emph{inner derivation} induced by $'a'$. An additive map $d$ is called a \emph{Jordan derivation} if $d(x^{2})=d(x)x+xd(x)$ for all $x \in R$. Obviously, every derivation is a Jordan derivation but the converse need not be true (see \cite[Example 3.2.1]{Ashraf2006}). Moreover, the question that ``when a Jordan derivation is a derivation?" caused a new and significant area of research (see \cite{Ashraf2000,Bresar88,bresar1988,herstein1957,Zalar1991}). In 1957, Herstein \cite{herstein1957} showed that for 2-torsion free prime rings , every Jordan derivation is an ordinary derivation. Later, Bre$\breve{s}$ar and Vukman \cite{Bresar88} gave a brief and elegant proof of this result. In the same year, Bre$\breve{s}$ar \cite{bresar1988} showed that for a rather wide class of rings, namely semiprime rings with 2-torsion free condition, every Jordan derivation is a derivation. Thenceforth a considerable amount of results has been proved in this direction.
\par
Let $R$ be a $\ast$-ring. An additive mapping $d:R \to R$ is called a \emph{$\ast$-derivation} if $d(xy)=d(x)y^{\ast}+xd(y)$ for all $x,y\in R$ and is called a \emph{Jordan $\ast$-derivation} if $d(x^{2})=d(x)x^{\ast}+xd(x)$ for all $x \in R$. The notions of $\ast$-derivation and Jordan $\ast$-derivation are first mentioned in \cite{bresar89}. Note that the mapping $x \to ax^{\ast}-xa$, where $a$ is a fixed element of $R$, is a Jordan $\ast$-derivation which is known as \emph{inner Jordan $\ast$-derivation}. The study of Jordan $\ast$-derivations has been originated from the problem of representability of quadratic forms by bilinear forms (see \cite{Semrl90,Semrl91}). Since then there has been a significant interest in the study of algebraic structure of Jordan $\ast$-derivations in rings and algebras, for a good cross-section we refer the reader to \cite{Ali2013,Bresar92,Lee2014,Lee2015}. For further generalizations and recent results, see \cite{Dar2021}.

\par Let $S$ be a subset of $R,$ a mapping $\varphi$ is called centralizing (resp. commuting) on $S$, if $[f(x),x] \in Z(R)$ (resp. $[f(x),x]=0$) for all $x \in S$.
The study of commuting and centralizing mappings goes back to 1955, when Divinsky \cite{Divinsky1955} proved that a simple artinian ring is commutative if it admits a commuting nontrivial automorphism. In this line of investigation, Posner  \cite{posner1957} proved another remarkable result which states that if there exists a nonzero centralizing derivation on $R,$ then $R$ must be commutative. Motivated by the centralizing and commuting mappings, Ali and Dar \cite{Ali2014} introduced $\ast$-centralizing and $\ast$-commuting mappings and defined as follows: a mapping $\varphi$ is called $\ast$-centralizing (resp. $\ast$-commuting) on a set $S$ if $[f(x),x^{\ast}] \in Z(R)$ (resp. $[f(x),x^{\ast}]=0$) for all $x\in S.$

\par There has been a rising literature on the investigation of centrally extended mappings in rings under various settings; for e.g. see \cite{bell18}, \cite{El-Deken2019}, \cite{El-Deken2020}, \cite{Muthana2020}. Continuing in this line of investigation, in this paper we introduce centrally extended Jordan derivations and give examples to show the existence of these maps in a $2$-torsion free prime rings. We also show that there exists no nonzero centrally extended Jordan derivation $\delta$ on a $2$-torsion free noncommutative prime ring (resp. prime ring with involution $'\ast'$) satisfying $[\delta(x),x]\in Z(R)$ (resp. $[\delta(x),x^{\ast}]\in Z(R)$) for all $x\in R,$ unless $R$ is an order in a central simple algebra of dimension $4$ over its center. Finally, we give the notion of centrally extended Jordan $\ast$-derivation and provide the analogous studies.

\section{Preliminaries}
By $s_{4},$ we denote the standard identity in four noncommuting variables, which is defined as follows:
\[
s_{4}(x_{1},x_{2},x_{3},x_{4})=\sum_{\sigma\in S_{4}}(-1)^{\sigma}x_{\sigma(1)}x_{\sigma(2)}x_{\sigma(3)}x_{\sigma(4)},
\]
where $S_{4}$ is the symmetric group of degree $4$ and $(-1)^{\sigma}$ is the sign of permutation $\sigma\in S_{4}.$ It is known by the standard PI-theory that, a prime ring $R$ satisfying $s_{4}$ can be characterized in a number of ways, as follows:
\begin{lemma}\cite[Lemma 1]{Bresar93}\label{F-1}
Let $R$ be a prime ring with extended centroid $C$. Then the following statements are equivalent:
\begin{itemize}
	\item[(i)] $R$ satisfies $s_{4}.$
	\item[(ii)] $R$ is commutative or $R$ embeds into $M_{2}(F),$ for a field $F.$
	\item[(iii)] $R$ is algebraic of bounded degree $2$ over $C$ (i.e., for any $a\in R,$ there exists a polynomial $x^{2}+\alpha x+\beta\in C[x]$ satisfied by $a$).
	\item[(iv)] $R$ satisfies $[[x^{2},y],[x,y]].$
\end{itemize}
\end{lemma}

\begin{lemma}\cite[Lemma 2.2]{Ali2014}\label{Fact-2}
Let $R$ be a $2$-torsion free semiprime ring with involution $'\ast'$. If an additive self-mapping $f$ of $R$ satisfies $[f(x),x^{\ast}]\in Z(R)$ for all $x\in R,$ then $[f(x),x^{\ast}]=0$ for all $x\in R.$
\end{lemma}

\begin{lemma}\cite[Proposition 2.1.7 (ii)]{Beidar1996}\label{L-2}
Let $R$ be a prime ring, $Q_{mr}(R)$ be the maximal right ring of quotients of $R$ and $\mathcal{D}$ be the set of all right dense ideals of $R.$ Then for all $q\in Q_{mr}(R),$ there exists $J\in \mathcal{D}$ such that $qJ\subseteq R.$
\end{lemma}

\begin{lemma}\cite[Proposition 3.1]{bresar1993}\label{Fact-1}
Let $R$ be a $2$-torsion free semiprime ring and $U$ be a Jordan subring of $R.$ If an additive self-mapping $f:R\to R$ satisfies $[f(x),x]\in Z(R)$ for all $x\in U,$ then $[f(x),x]=0$ for all $x\in U.$
\end{lemma}

\begin{lemma}\cite[Theorem 3.2]{bresar1993}\label{L-4}
Let $R$ be a prime ring. If an additive mapping $F:R\to R$ is commuting on $R,$ then there exists $\lambda\in C$ and an additive $\xi:R\to C,$ such that $F(x)=\lambda x+\xi(x)$ for all $x\in R.$
\end{lemma}

\begin{lemma}\cite[Theorem 6.5.1]{Herstein76}\label{L-5}
If $R$ is a semiprime ring, then $\overline{H(R)}$ is semiprime, where $\overline{H(R)}$ is the ring generated by all symmetric elements in $R$.
\end{lemma}

\begin{lemma}\cite[Theorem 6.5.3]{Herstein76}\label{L-6}
If $R$ is a semiprime ring, then $Z(\overline{H(R)})\subseteq Z(R).$
\end{lemma}

\begin{lemma}\cite[Lemma 2]{lanski1976}\label{L-7}
If $R$ is a semiprime ring and $[S(R)^{2},S(R)^{2}]=(0),$ then $R$ satisfies $s_{4}.$
\end{lemma}

\begin{lemma}\cite[Theorem 3]{lee1986}\label{L-8}
Let $R$ be a prime ring with involution $'\ast'$ and center $Z(R).$ If $n$ be a fixed natural number such that $x^{n}\in Z(R)$ for all $x\in H(R),$ then $R$ satisfies $s_{4}.$
\end{lemma}

\begin{lemma}\cite[Theorem 7]{lee1986}\label{L-9}
Let $R$ be a prime ring with involution $'\ast'$ and center $Z(R).$ If $d$ is a nonzero derivation on $R$ such that $d(x)x+xd(x)\in Z(R)$ for all $x\in S(R),$ then $R$ satisfies $s_{4}.$
\end{lemma}

\begin{lemma}\cite[Lemma 1.3]{Zalar1991}\label{L-10}
Let $R$ be a semiprime ring and $a\in R$ some fixed element. If $a[x,y]=0$ for all $x,y\in R,$ then there exists an ideal $I$ of $R$ such that $a\in I\subseteq Z(R)$ holds.
\end{lemma}

\section{Results on centrally extended Jordan derivations}
 This section deals with the study of centrally extended Jordan derivations of rings. In fact, we characterize $2$-torsion free noncommutative prime rings admitting $CE-$Jordan derivations.  Recently, Bell and Daif \cite{bell18} introduced \emph{centrally extended derivations} which are obviously a generalization of derivations; and they discussed the existence of these mappings in rings. Accordingly, a self-mapping $d$ of $R$ is called a centrally extended derivation if $d(x+y)-d(x)-d(y)\in Z(R)$ and $d(xy)-d(x)y-xd(y)\in Z(R)$ for all $x,y \in R.$ Motivated by this, we now introduce a finer notion than centrally extended derivation, and call a \emph{centrally extended Jordan derivation}, as follows:
\begin{definition}
A mapping $\delta:R\to R$ that satisfies
\begin{equation}\tag{A}
\delta(x+y)-\delta(x)-\delta(y) \in Z(R),
\end{equation}
\begin{equation}\tag{B}
\delta(x\circ y)-\delta(x)\circ y-x\circ \delta(y) \in Z(R)
\end{equation}
 for all $x,y \in R,$ is called a \emph{centrally extended Jordan derivation} of $R$. We shall abbreviate this map as $CE-$Jordan derivation.
\end{definition}

\begin{example}\label{example-1}
Let $R=M_{2}(\mathbb{Z})\times\mathbb{Z}$ be a ring and define a mapping $\delta:R\to R$ by
$$\delta\left(\left(\begin{array}{cc}
a & b \\
c & d
\end{array} \right),x\right)=\left(\left(\begin{array}{cc}
0 & -b \\
c & 0
\end{array} \right),1\right).$$ Then, it is straightforward to check that $\delta$ is a $CE-$Jordan derivation of $R.$
\end{example}

\begin{remark}\label{rem-1}
If $R$ is a 2-torsion free noncommutative prime ring, then it is not difficult to see that an additive map $\delta:R\to R$ is a $CE-$Jordan derivation if and only if $\delta(x^{2})-\delta(x)x-x\delta(x)\in Z(R)$ for all $x\in R.$ It is natural to ask whether a $CE-$Jordan derivation can be a $CE-$derivation or a Jordan derivation? In the following counter-example, we show that in case $R$ is a noncommutative prime ring, the above statement is not always true: \\
Let $\mathbb{Z}$ be the ring of integers and $$R=\left\{\left(\begin{array}{cc}
a & b \\
c & d
\end{array} \right) |~a,b,c,d \in \mathbb{Z}\right\}$$ be a noncommutative prime ring.
Then a mapping $\delta:R\to R$ such that $$\delta\left(\begin{array}{cc}
a & b \\
c & d
\end{array} \right)=\left( \begin{array}{cc}
0 & b \\
b & 0
\end{array} \right)$$ is a $CE-$Jordan derivation, but neither a $CE-$derivation nor a Jordan derivation.
\end{remark}

 We begin with the following lemma:

\begin{lemma}\label{lemma-1}
Let $R$ be a 2-torsion free ring with no nonzero central ideal. If $\delta$ is a $CE-$Jordan derivation of $R,$ then $\delta$ is additive.
\end{lemma}
\begin{proof}
  Let $\delta$ be a $CE-$Jordan derivation of $R$. In view of condition $(A),$ for any $x,y,z\in R,$ we have
		\begin{equation}\label{eq-1}
		\delta(x+y)=\delta(x)+\delta(y)+c_{\delta(x,y,+)},
		\end{equation}
		where $c_{\delta(x,y,+)}\in Z(R).$ There exists some $c_{\delta(z,x+y,\circ)}\in Z(R)$ such that
		\begin{eqnarray}\label{eq-2}
		\delta(z\circ(x+y)) &=& \delta(z)\circ(x+y)+z\circ\delta(x+y)+c_{\delta(z,x+y,\circ)}\notag \\
		&=& \delta(z)\circ x+\delta(z)\circ y+z\circ(\delta(x)+\delta(y)+c_{\delta(x,y,+)})+c_{\delta(z,x+y,\circ)}\notag \\
		&=& \delta(z)\circ x+\delta(z)\circ y+z\circ\delta(x)+z\circ\delta(y)+2zc_{\delta(x,y,+)}\notag\\
&&+c_{\delta(z,x+y,\circ)}.
		\end{eqnarray}
		Another way of looking at it is,
		\begin{eqnarray}\label{eq-3}
		\delta(z\circ(x+y)) &=& \delta(z\circ x+z\circ y)\notag \\
		&=& \delta(z\circ x)+\delta(z\circ y)+c_{\delta(z \circ x,z \circ y,+)}\notag \\
		&=& \delta(z)\circ x+z\circ\delta(x)+c_{\delta(z, x,\circ)}+\delta(z)\circ y+z\circ\delta(y)\notag\\
		&&+c_{\delta(z ,y, \circ)}+c_{\delta(z \circ x,z \circ y,+)},
		\end{eqnarray}
		where $c_{\delta(z \circ x,z \circ y,+)},c_{\delta(z, x,\circ)}$ and $c_{\delta(z ,y, \circ)}$ are the central elements.
\par Comparing (\ref{eq-2}) and (\ref{eq-3}), we conclude that $2zc_{\delta(x,y,+)}+c_{\delta(z,x+y,\circ)}=c_{\delta(z \circ x,z \circ y,+)}+c_{\delta(z, x,\circ)}+c_{\delta(z ,y, \circ)}\in Z(R)$. It forces that $Rc_{\delta(x,y,+)}\subseteq Z(R),$ where $c_{\delta(x,y,+)}$ is a fixed central element in $R,$ but $R$ has no nonzero central ideal, therefore $Rc_{\delta(x,y,+)}=(0).$ Likewise, we get $c_{\delta(x,y,+)}R=(0).$ It implies that $c_{\delta(x,y,+)}\in A(R),$ the annihilator of $R.$ But $A(R)$ is always a central ideal in $R,$ hence our hypothesis forces $A(R)=(0)$ and consequently $c_{\delta(x,y,+)}=0.$ From (\ref{eq-1}), we get $\delta(x+y)=\delta(x)+\delta(y)$ for all $x,y\in R,$ as desired.
\end{proof}

\begin{corollary}\label{cor-01}
Let $R$ be a 2-torsion free noncommutative prime ring. If $\delta$ is a $CE-$Jordan derivation of $R,$ then $\delta$ is additive.
\end{corollary}

Now, we are in position to state and prove the first result of this paper.
\begin{theorem}\label{Theorem-1}
Let $R$ be a $2$-torsion free noncommutative prime ring. If $R$ admits a nonzero $CE-$Jordan derivation $\delta:R\to R$ such that $[\delta(x),x]\in Z(R)$ for all $x\in R,$ then either $\delta=0$ or $R$ is an order in a central simple algebra of dimension at most $4$ over its center.
\end{theorem}
\begin{proof}
By the hypothesis, we have $[\delta(x),x]\in Z(R)$ for all $x\in R.$ In view of Corollary \ref{cor-01}, $\delta$ is additive and hence from Lemma \ref{Fact-1}, it follows that
\begin{equation}\label{eq-4}
[\delta(x),x]=0 \mbox{~for~all~}x\in R.
\end{equation}
Since $\delta$ is an additive and commuting function, by Lemma \ref{L-4}, there exists $\lambda \in C$ (extended centroid of $R$) and an additive mapping $\sigma : R \to C$ such that
\begin{equation}\label{eq-5}
\delta(x)=\lambda x +\sigma(x)\mbox{~for~all~}x\in R.
\end{equation}
Polarizing (\ref{eq-4}), we have
\[
[\delta(x),y]+[\delta(y),x]=0\mbox{~for~all~}x,y\in R.
\]
Replacing $y$ by $x \circ y,$ we get
\[
[\delta(x), x \circ y]+[\delta(x \circ y),x]=0\mbox{~for~all~}x,y\in R.
\]
It implies
\[
[\delta(x), x \circ y]+[\delta(x) \circ y,x]+[x \circ \delta(y),x]=0\mbox{~for~all~}x,y\in R.
\]
Using (\ref{eq-5}) in the preceding relation to get
\[
[\lambda x+\sigma(x), x \circ y]+[(\lambda x+\sigma(x)) \circ y,x]+[x \circ (\lambda y+\sigma(y)),x]=0,
\]
that is
\begin{equation}\label{eq-6}
[\lambda x, x \circ y]+[\lambda(x\circ y),x]+[\sigma(x) \circ y, x]+[\lambda(x\circ y),x]+[x\circ\sigma(y),x]=0\mbox{~for~all~}x,y\in R.
\end{equation}
It follows that $R$ satisfies
\begin{equation}\label{eq-7}
\lambda [ x, x \circ y]+\lambda [ x \circ y,x]+2\sigma(x)[ y, x]+\lambda [x \circ y,x]=0,
\end{equation}
and so
\begin{equation}\label{eq-8}
2\sigma(x)[ y, x]+\lambda [x \circ y,x]=0\mbox{~for~all~}x,y\in R.
\end{equation}
Further, it implies
\[
2\sigma(x)[y,x]+\lambda [y,x^{2}]=0\mbox{~for~all~}x,y\in R,
\]
and a fortiori
\[
\lambda [[y,x^{2}],[y,x]]=0\mbox{~for~all~}x,y\in R.
\]
It implies that either $\lambda=0$ or $[[y,x^{2}],[y,x]]=0$ for all $x,y\in R.$ By Lemma \ref{F-1}, the latter case is equivalent to the $s_{4}$ identity and $R$ is assumed to be noncommutative, therefore $R$ is an order in a central simple algebra of dimension at most 4 over $Z(R).$
\par On the other hand, let us assume that $\lambda=0.$ Then from (\ref{eq-8}), we have
\[
2\sigma(x)[y,x]=0\mbox{~for~all~}x,y\in R.
\]
Using the restriction on torsion of $R,$ we have
\[
\sigma(x)[y,x]=0\mbox{~for~all~}x,y\in R.
\]
Since $R$ is a prime ring, for each $x\in R,$ either $\sigma(x)=0$ or $[R,x]=(0).$
Put $\mathfrak{U}=\{x \in R:~\sigma(x)=0\}$ and $\mathfrak{V}=\{x \in R:~[R,x]=(0)\}.$ Therefore, we note that $R$ can be written as the set-theoretic union of the additive subgroups $\mathfrak{U}$ and $\mathfrak{V},$ which is not possible. Thus, we have either $R=\mathfrak{U}$ or $R =\mathfrak{V}.$ It implies that either $\sigma(x)=0$ for all $x\in R$ or $[R,x]=(0)$ for all $x\in R$. If $\sigma(x)=0$ for all $x\in R,$ then from (\ref{eq-5}), we find $\delta(x)=0$ for all $x \in R$. In the other case $R$ is a commutative ring; which leads a contradiction. This completes the proof.
\end{proof}


\begin{theorem}\label{Theorem-2}
Let $R$ be a 2-torsion free noncommutative prime ring with involution $'\ast'$ that admits a $CE-$Jordan derivation $\delta:R\to R$ such that $[\delta(x),x^{\ast}]\in Z(R)$ for all $x\in R.$ Then either $\delta=0$ or $R$ is an order in a central simple algebra of dimension at most $4$ over its center.
\end{theorem}
\begin{proof}
Let us assume that $[\delta(x),x^{\ast}]\in Z(R)$ for all $x\in R.$ With the aid of Corollary \ref{cor-01} and Lemma \ref{Fact-2}, we have
\begin{equation}\label{eq-9}
[\delta(x),x^{\ast}]=0\mbox{~for~all~}x\in R.
\end{equation}
Applying involution in (\ref{eq-9}), we get
\begin{equation}\label{eq-10}
[\delta(x)^{\ast},x]=0\mbox{~for~all~}x\in R.
\end{equation}
In view of Lemma \ref{L-4}, there exists $\lambda\in C$ and an additive mapping $\sigma:R\to C$ such that
\[
\delta(x)^{\ast}=\lambda x+\sigma(x)\mbox{~for~all~}x\in R.
\]
It implies
\begin{equation}\label{eq-11}
\delta(x)=\lambda^{\ast}x^{\ast}+\sigma(x)^{\ast}\mbox{~for~all~}x\in R.
\end{equation}
We now split the proof into two parts, as follows:\\
\textbf{Case 1.} Suppose that there exists a non-zero element $c\in C$ such that $c^{\ast}\neq c.$ Let $ c^{\ast}-c=z_{c}.$ Clearly $z_{c}^\ast=-z_{c}\neq 0.$
Also $z_{c}\in C$. By Lemma \ref{L-2}, there exists a nonzero ideal $J$ of $R$ such that $z_{c}J \subseteq R$.
Polarizing (\ref{eq-9}), we get
\begin{equation}\label{eq-12}
[\delta(x),y^{\ast}]+[\delta(y),x^{\ast}]=0\mbox{~for~all~}x,y\in R.
\end{equation}
Replacing $y$ by $y \circ r$ in (\ref{eq-12}), where $r \in J,$ we get
\begin{equation}\label{eq-13}
[\delta(x),(y \circ r)^{\ast}]+[\delta(y \circ r),x^{\ast}]=0.
\end{equation}
Therefore, we have
\begin{equation}\label{eq-14}
[\delta(x),y^{\ast} \circ r^{\ast}]+[\delta(y) \circ r,x^{\ast}]+[y \circ \delta(r),x^{\ast}]=0\mbox{~for~all~}x,y\in R, r \in J.
\end{equation}
In particular for $y=h\in H(R)$ in (\ref{eq-14}), we get
\[
[\delta(x),r^{\ast} \circ h]+[r \circ \delta(h),x^{\ast}]+[\delta(r) \circ h,x^{\ast}]=0.
\]
Using (\ref{eq-11}) in the above expression, we obtain
\[
\lambda^{\ast}[x^{\ast},r^{\ast} \circ h]+\lambda^{\ast}[r^{\ast} \circ h,x^{\ast}]+\lambda^{\ast}[r \circ h,x^{\ast}]+2\sigma(r)^{\ast}[h,x^{\ast}]+2\sigma(h)^{\ast}[r,x^{\ast}]=0.
\]
It implies
\begin{equation}\label{eq-15}
\lambda^{\ast}[r \circ h,x^{\ast}]+2\sigma(r)^{\ast}[h,x^{\ast}]+2\sigma(h)^{\ast}[r,x^{\ast}]=0\mbox{~for~all~}x\in R, r \in J, h \in H(R).
\end{equation}
Replacing $r$ by $z_{c}r=rz_{c}$ in (\ref{eq-15}), we get
\begin{equation}\label{eq-16}
\lambda^{\ast}[r \circ h,x^{\ast}]z_{c}+2\sigma(rz_{c})^{\ast}[h,x^{\ast}]+2\sigma(h)^{\ast}[r,x^{\ast}]z_{c}=0\mbox{~for~all~}x\in R, r \in J, h \in H(R).
\end{equation}
Multiply (\ref{eq-15}) by $z_{c}$ and comparing with (\ref{eq-16}), we obtain
\[
(\sigma(rz_{c})^{\ast}-\sigma(r)^{\ast}z_{c})[h,x^{\ast}]=0.
\]
Therefore primeness of $R$ implies that either $[h,x^{\ast}]=0$ for all $h \in H(R), x \in R$ or $\sigma(rz_{c})^{\ast}=\sigma(r)^{\ast}z_{c}$. In the former case, $R$ satisfies $s_{4}$ identity by Lemma \ref{L-8}. But $R$ is noncommutative, therefore $R$ is an order in a central simple algebra of dimension at most $4$ over its center. In the latter case, we have
\begin{equation}\label{eq-17}
\sigma(rz_{c})^{\ast}=\sigma(r)^{\ast}z_{c}.
\end{equation}
Replacing $r$ by $rz_{c}$ in (\ref{eq-14}), we obtain
\[
[\delta(x),y^{\ast}\circ r^{\ast}]z_{c}^{\ast}+[\delta(y)\circ r,x^{\ast}]z_{c}+[y\circ \delta(rz_{c}),x^{\ast}]=0.
\]
It implies
\begin{equation}\label{eq-18}
-[\delta(x),y^{\ast}\circ r^{\ast}]z_{c}+[y\circ \delta(rz_{c}),x^{\ast}]+[\delta(y)\circ r,x^{\ast}]z_{c}=0\mbox{~for~all~}x,y\in R, r \in J.
\end{equation}
Multiplying (\ref{eq-14}) by $z_c$ and then adding it into (\ref{eq-18}), we find
\[
[y\circ \delta(rz_{c}),x^{\ast}]+[y\circ \delta(r),x^{\ast}]z_{c}+2[\delta(y)\circ r,x^{\ast}]z_{c}=0\mbox{~for~all~}x,y\in R, r \in J,
\]
which by virtue of (\ref{eq-11}), leads to
\[
[y\circ (\lambda^{\ast}r^{\ast}z_{c}^{\ast}+\sigma(rz_{c})^{\ast}),x^{\ast}]+[y\circ (\lambda^{\ast}r^{\ast}+\sigma(r)^{\ast}),x^{\ast}]z_{c}+2[\delta(y)\circ r,x^{\ast}]z_{c}=0.
\]
An application of (\ref{eq-17}) yields
\[
[y\circ (-\lambda^{\ast}r^{\ast}+\sigma(r)^{\ast}),x^{\ast}]z_{c}+[y\circ (\lambda^{\ast}r^{\ast}+\sigma(r)^{\ast}),x^{\ast}]z_{c}+2[\delta(y)\circ r,x^{\ast}]z_{c}=0.
\]
It follows that
\[
2[y\circ \sigma(r)^{\ast},x^{\ast}]z_{c}+2[\delta(y)\circ r ,x^{\ast}]z_{c}=0\mbox{~for~all~}x,y\in R, r \in J.
\]
Since $z_{c} \neq 0,$ it implies that
\[
2(2\sigma(r)^{\ast}[y,x^{\ast}]+[r \circ \delta(y),x^{\ast}])=0\mbox{~for~all~}x,y\in R, r \in J.
\]
Using 2-torsion freeness of $R$ and (\ref{eq-11}), we have
\[
2\sigma(r)^{\ast}[y,x^{\ast}]+[r \circ (\lambda^{\ast}y^{\ast}+\sigma(y)^{\ast}),x^{\ast}]=0.
\]
It implies
\begin{equation}\label{eq-19}
2\sigma(r)^{\ast}[y,x^{\ast}]+\lambda^{\ast}[r \circ y^{\ast},x^{\ast}]+2\sigma(y)^{\ast}[r,x^{\ast}]=0\mbox{~for~all~}x,y\in R, r \in J.
\end{equation}
In particular, for $y\in J,$ we replace $y$ by $yz_{c}$ and get
\[
2\sigma(r)^{\ast}[y,x^{\ast}]z_{c}-\lambda^{\ast}[r \circ y^{\ast},x^{\ast}]z_{c}+2\sigma(yz_{c})^{\ast}[r,x^{\ast}]=0.
\]
Using (\ref{eq-17}), we have
\begin{equation}\label{eq-20}
2\sigma(r)^{\ast}[y,x^{\ast}]z_{c}-\lambda^{\ast}[r \circ y^{\ast},x^{\ast}]z_{c}+2\sigma(y)^{\ast}[r,x^{\ast}]z_{c}=0.
\end{equation}
Multiplying (\ref{eq-19}) by $z_{c}$ and then subtract it from (\ref{eq-20}), we conclude that
\[
\lambda^{\ast}[r \circ y^{\ast},x^{\ast}]z_{c}=0\mbox{~for~all~}x\in R,~r,y\in J.
\]
By the primeness of $R$ we have either $[r \circ y,x]=0$ or $\lambda=0$. Since $R$ is noncommutative, the first situation cannot occur; consequently $\lambda=0.$ Thus $\delta(x)=\sigma(x)^{\ast}$ for all $x\in R.$ From relation (B), we have
\[
\sigma(x \circ y)^{\ast}-\sigma(x)^{\ast} \circ y- x \circ \sigma(y)^{\ast} \in Z(R)\mbox{~for~all~}x,y\in R,
\]
that is
\[
2\sigma(x)^{\ast}[y,x]=0\mbox{~for~all~}x,y\in R.
\]
Using $2$-torsion freeness hypothesis, we find
\[
\sigma(x)^{\ast}[y,x]=0\mbox{~for~all~}x,y\in R.
\]
Further proceeding as Theorem \ref{Theorem-1}, we are done in this case.\\
\textbf{Case 2.} Let $c^{\ast}=c$ for all $c\in C.$ Polarizing (\ref{eq-9}), we have
\begin{equation}\label{eq-21}
[\delta(x),y^{\ast}]+[\delta(y),x^{\ast}]=0\mbox{~for~all~}x,y\in R.
\end{equation}
Replacing $y$ by $h\circ k,$ where $k\in S(R)$ and $h \in H(R),$ we have
\begin{equation}\label{eq-22}
-[\delta(x),h \circ k]+[\delta(h) \circ k,x^{\ast}]+[h \circ \delta(k),x^{\ast}]=0.
\end{equation}
In view of (\ref{eq-11}), it follows that
\[
-[\lambda x^{\ast},h \circ k]+[(\lambda h+\sigma(h))\circ k,x^{\ast}]+ [h \circ (-\lambda k+\sigma(k)),x^{\ast}]=0.
\]
It implies
\begin{equation}\label{eq-23}
-\lambda [x^{\ast},h \circ k]+2\sigma(h)[k,x^{\ast}]+2\sigma(k)[h,x^{\ast}]=0\mbox{~for~all~}x\in R, h \in H(R), k \in S(R).
\end{equation}
Replacing $x$ by $k$ in (\ref{eq-23}), we find
\begin{equation}\label{eq-24}
\lambda [k^{2},h]-2\sigma(k)[h,k]=0\mbox{~for~all~}h \in H(R), k \in S(R).
\end{equation}
Taking involution on both side in (\ref{eq-24}) and using $z^{\ast}=z$ for all $z \in C,$ we find
\begin{equation}\label{eq-25}
\lambda [k^{2},h]+2\sigma(k)[h,k]=0\mbox{~for~all~}h \in H(R), k \in S(R).
\end{equation}
Adding (\ref{eq-25}) and (\ref{eq-24}), we obtain
\begin{equation}\label{eq-26}
2\lambda [k^{2},h]=0\mbox{~for~all~}h\in H(R), k \in S(R).
\end{equation}
\par Thereby 2-torsion freeness and primeness of $R$ implies that either $\lambda=0$ or $[k^{2},h]={0}$ for all $k\in S(R)$ and $h\in H(R).$ Latter case implies $d_{h}(k)k+kd_{h}(k)=0$ for all $k\in S(R),$ where $d_{h}$ is the inner derivation induced by $h.$ In view of Lemma \ref{L-9}, $R$ satisfies $s_{4}$ or $d_{h}=0.$
Therefore either $R$ satisfies $s_{4}$ or $H(R)\subseteq Z(R).$ By Lemma \ref{L-8}, in each of the situation, $R$ satisfies $s_{4}$. But $R$ is noncommutative, hence $R$ is an order in a central simple algebra of dimension at most $4$ over its center.
\par We now consider $\lambda=0.$ Using (B) and (\ref{eq-11}), we have
\[
\sigma(x \circ y)-\sigma(x) \circ y- x \circ \sigma(y) \in Z(R)\mbox{~for~all~}x,y\in R.
\]
Therefore
\[
2\sigma(x)[y,x]=0\mbox{~for~all~}x,y\in R.
\]
Further proceeding as the proof of Theorem \ref{Theorem-1}, we get the conclusion.
\end{proof}

\section{Results on centrally extended Jordan $\ast$-derivation}
Let $R$ be a ring with involution $'\ast'.$ In \cite{El-Deken2019}, El-Deken and Nabiel introduced the notion of \emph{centrally extended $\ast$-derivation} and investigated the case when centrally extended $\ast$-derivations are $\ast$-derivations. More specifically, they established the following result: \emph{If R is a semiprime $\ast$-ring with no nonzero central ideals, then every centrally extended $\ast$-derivation $d$ on $R$ is a $\ast$-derivation}. Motivated by the concept of centrally extended $\ast$-derivations, we now introduce the notion of centrally extended Jordan $\ast$-derivation as follows:

\begin{definition}
A mapping $\delta:R\to R$ that satisfies
\begin{equation}\tag{C}
\delta(x+y)-\delta(x)-\delta(y)\in Z(R),
\end{equation}
\begin{equation}\tag{D}
\delta(x\circ y)-\delta(x)y^{\ast}-x\delta(y)-\delta(y)x^{\ast}-y\delta(x)\in Z(R)
\end{equation}
for all $x,y \in R$ is called \emph{centrally extended Jordan $\ast$-derivation} of $R$. We shall abbreviate this map as $CE-$Jordan $\ast$-derivation.
\end{definition}

\begin{example}\label{example-2}
Let $R=\bigg\{\left(
           \begin{array}{ccc}
             0 & a & b \\
             0 & 0 & c \\
             0 & 0 & 0 \\
           \end{array}
         \right):~a,b,c\in\mathbb{Z}
\bigg\}$ be a ring. Define mapping $\delta,\ast:R\to R$ by
$$\delta\left(\begin{array}{ccc}
0 & a & b\\
0 & 0 & c\\
0 & 0 & 0\\
\end{array} \right)=\left(\begin{array}{ccc}
0 & 0 & k\\
0 & 0 & 0\\
0 & 0 & 0\\
\end{array} \right),
\left(\begin{array}{ccc}
0 & a & b\\
0 & 0 & c\\
0 & 0 & 0\\
\end{array} \right)^{\ast}=\left(\begin{array}{ccc}
0 & c & b\\
0 & 0 & a\\
0 & 0 & 0\\
\end{array} \right)
$$ where $k$ is a fixed integer. Then, it is straight forward to check that $\delta$ is a $CE-$Jordan $\ast$-derivation of $R,$ with involution $'\ast'$ of $R.$
\end{example}

\begin{remark}\label{rem-2}
If $R$ is a 2-torsion free noncommutative prime ring with involution $'\ast'$, then an additive map $d$ is a $CE-$Jordan $\ast$-derivation if and only if $d(x^{2})-d(x)x^{\ast}-xd(x)\in Z(R)$ for all $x\in R.$ It is observed that, in this case, a $CE-$Jordan $\ast$-derivation is not necessarily a Jordan $\ast$-derivation or $CE$ $\ast$-derivation, for example:
\\Let $\mathbb{Z}$ be the ring of integers and $R=\left\{ \left(\begin{array}{cc}
a & b \\
c & d
\end{array} \right) |~a,b,c,d \in \mathbb{Z} \right\},$ a prime ring. Define a mapping $\ast:R\to R$ by
$\left( \begin{array}{cc}
a & b \\
c & d
\end{array} \right)^\ast= \left( \begin{array}{cc}
d & -b \\
-c & a
\end{array} \right),$ which is an involution of $R$.
If $\delta:R\to R$ be a mapping such that $\delta\left(\begin{array}{cc}
a & b \\
c & d
\end{array} \right)=\left( \begin{array}{cc}
a & 0 \\
0 & a
\end{array} \right),$ then $\delta$ is a $CE-$Jordan $\ast$-derivation but neither Jordan $\ast$-derivation nor $CE$ $\ast$-derivation.
\end{remark}

\par In this section our focus will be on $CE-$Jordan $\ast$-derivations acting on 2-torsion free noncommutative prime rings. More specifically, we prove the $CE-$Jordan $\ast$-derivation analogy of the above results. We begin our discussions with the following lemma.

\begin{lemma}\label{lemma-2}
Let $R$ be a 2-torsion free semiprime ring with involution $'\ast'$ and with no nonzero central ideal. If $\delta$ is a $CE-$Jordan $\ast$-derivation of $R,$ then $\delta$ is additive.
\end{lemma}
\begin{proof}
	For any $x,y,z\in R,$ in view of (C), it follows that there exists $c_{\delta(x,y,+)}\in Z(R)$ such that
	\begin{equation}\label{eq-27}
		\delta(x+y)=\delta(x)+\delta(y)+c_{\delta(x,y,+)},
	\end{equation}
	Using (D), we have
	\begin{eqnarray}\label{eq-28}
		\delta(z\circ(x+y)) &=& \delta(z)(x^{\ast}+y^{\ast})+\delta(x+y)z^{\ast}+(x+y)\delta(z)\notag\\
		&&+z\delta(x+y)+c_{\delta(z,x+y,\circ)} \notag \\
		&=&   \delta(z)(x^{\ast}+y^{\ast})+\bigg(\delta(x)+\delta(y)+c_{\delta(x,y,+)}\bigg)z^{\ast}+(x+y)\notag\\
		&&\delta(z)+z\bigg(\delta(x)+\delta(y)+c_{\delta(x,y,+)}\bigg)+c_{\delta(z,x+y,\circ)}\notag\\
		&=& \delta(z)(x^{\ast}+y^{\ast})+(\delta(x)+\delta(y))z^{\ast}+(x+y)\delta(z)\notag\\
		&&+z(\delta(x)+\delta(y))+c_{\delta(x,y,+)}(z+z^{\ast})+c_{\delta(z,x+y,\circ)},
	\end{eqnarray}
	where $c_{\delta(z,x+y,\circ)},c_{\delta(x,y,+)}$ are central elements.
	Computing in a different manner, we get
	\begin{eqnarray}\label{eq-29}
		\delta(z\circ(x+y)) &=& \delta(z\circ x+z\circ y)\notag \\
		&=& \delta(zx+xz)+\delta(zy+yz)+c_{\delta(z\circ x,z \circ y,+)}\notag \\
		&=& \bigg(\delta(z)x^{\ast}+\delta(x)z^{\ast}+z\delta(x)+x\delta(z)+c_{\delta(z,x,\circ)}\bigg)+\bigg(\delta(z)y^\ast+\delta(y)z^{\ast}+z\notag\\
		&&\delta(y)+y\delta(z)+c_{\delta(z,y,\circ)}\bigg)+c_{\delta(z\circ x,z \circ y,+)},
	\end{eqnarray}
	where $c_{\delta(z\circ x,z \circ y,+)},c_{\delta(z,x,\circ)},c_{\delta(z,y,\circ)}$ are central elements.
	\par Comparing (\ref{eq-28}) and (\ref{eq-29}) to conclude $(z+z^{\ast})c_{\delta(x,y,+)}+c_{\delta(z,x+y,\circ)}=c_{\delta(z,x,\circ)}+c_{\delta(z,y,\circ)}+c_{\delta(z\circ x,z \circ y,+)}.$ It forces $(z+z^{\ast})c_{\delta(x,y,+)}$ for all $z\in R,$ where $c_{\delta(x,y,+)}$ is a fixed central element in $Z(R).$ Replacing $z$ by $h,$ where $h \in H(R)$, we have $2hc_{\delta(x,y,+)} \in Z(R).$ Using 2-torsion freeness hypothesis, we get
	\begin{equation}{\label{eq-30}}
		hc_{\delta(x,y,+)} \in Z(R)\mbox{~for~all~}h \in H(R).
	\end{equation}
	It implies
	\[
	c_{\delta(x,y,+)}[h,u]=0\mbox{~for~all~}h\in H(R),~u \in R.
	\]
	Replacing $h$ by $hh',$ where $h' \in H(R),$ we get
	\begin{equation*}
		c_{\delta(x,y,+)}[hh^{'},u] = c_{\delta(x,y,+)}[h,u]h^{'}+c_{\delta(x,y,+)}h[h^{'},u]=0.
	\end{equation*}
	Therefore, it follows that
	\begin{equation}\label{eq-31}
		c_{\delta(x,y,+)}[\overline{H(R)},u]=(0)\mbox{~for~all~}u \in R,
	\end{equation}
	where $\overline{H(R)}$ is the ring generated by $H(R)$.
	In particular
	\begin{equation*}
		c_{\delta(x,y,+)}[\overline{H(R)},\overline{H(R)}]=(0).
	\end{equation*}
	In view of Lemma \ref{L-5}, $\overline{H(R)}$ is also a semi prime ring. Now, we shall first show that $c_{\delta(x,y,+)}\in H(R).$ Suppose that $c_{\delta(x,y,+)} \not\in H(R),$ i.e., $c_{\delta(x,y,+)}^{\ast} \neq c_{\delta(x,y,+)}.$ Then $(c_{\delta(x,y,+)}^{\ast}-c_{\delta(x,y,+)})k \in H(R)$ for all $k \in S(R)$. Using (\ref{eq-30}), we get
	\begin{equation}\label{eq-32}
		c_{\delta(x,y,+)}(c_{\delta(x,y,+)}^{\ast}-c_{\delta(x,y,+)})k \in Z(R)\mbox{~for~all~}k \in S(R).
	\end{equation}
	Applying involution, we have
	\begin{equation}\label{eq-33}
		c_{\delta(x,y,+)}^{\ast}(c_{\delta(x,y,+)}^{\ast}-c_{\delta(x,y,+)})k \in Z(R)\mbox{~for~all~}k \in S(R).
	\end{equation}
	Subtracting $(\ref{eq-32})$ from $(\ref{eq-33})$, we get
	\begin{equation}\label{eq-34}
		(c_{\delta(x,y,+)}^{\ast}-c_{\delta(x,y,+)})^{2}k \in Z(R)\mbox{~for~all~}k \in S(R).
	\end{equation}
	Using (\ref{eq-30}), we find  $c_{\delta(x,y,+)}^{\ast}h \in Z(R)$ for all $h \in H(R)$.
	An easy consequence of this is
	\begin{equation}\label{eq-35}
		(c_{\delta(x,y,+)}^{\ast}-c_{\delta(x,y,+)})h \in Z(R)\mbox{~for~all~}h \in H(R).
	\end{equation}
	Since $c_{\delta(x,y,+)}\in Z(R)$, left multiplying (\ref{eq-35}) by $c_{\delta(x,y,+)}^{\ast}-c_{\delta(x,y,+)}$ in order to find
	\begin{equation}\label{eq-36}
		(c_{\delta(x,y,+)}^{\ast}-c_{\delta(x,y,+)})^{2}h \in Z(R)\mbox{~for~all~}h \in H(R).
	\end{equation}
	Adding $(\ref{eq-34})$ and $(\ref{eq-36})$, we have
	\begin{equation}\label{eq-37}
		(c_{\delta(x,y,+)}^{\ast}-c_{\delta(x,y,+)})^{2}(h+k) \in Z(R)\mbox{~for~all~}h \in H(R),~k \in S(R).
	\end{equation}
	It yields
	\begin{equation*}
		(c_{\delta(x,y,+)}^{\ast}-c_{\delta(x,y,+)})^{2}R \subseteq Z(R).
	\end{equation*}
	Since $R$ does not contain non-zero central ideals, thus we have $(c_{\delta(x,y,+)}^{\ast}-c_{\delta(x,y,+)})^{2}=0,$ where $c_{\delta(x,y,+)}^{\ast}-c_{\delta(x,y,+)} \in Z(R)$. The fact that center of a semiprime ring has no nonzero nilpotent elements, forces $c_{\delta(x,y,+)}^{\ast}-c_{\delta(x,y,+)}=0$, a contradiction to our assumption. Hence $c_{\delta(x,y,+)}\in H(R).$
	\par
	Therefore in view of Lemma \ref{L-10}, our expression $c_{\delta(x,y,+)}[\overline{H(R)},\overline{H(R)}]=(0)$ implies that there exists an ideal $I \subseteq \overline{H(R)}$ such that $c_{\delta(x,y,+)} \in I \subseteq Z(\overline{H(R)})$. Using Lemma \ref{L-6}, we get  $c_{\delta(x,y,+)} \in I \subseteq Z(R)$. As $R$ does not contain non-zero central ideals, we get $c_{\delta(x,y,+)}=0;$ and hence $\delta$ is additive, by (\ref{eq-27}). It completes the proof.
\end{proof}

\begin{corollary}\label{cor-02}
Let $R$ be a 2-torsion free noncommutative prime ring with involution $'\ast'$. If $\delta$ is a $CE-$Jordan $\ast$-derivation of $R,$ then $\delta$ is additive.
\end{corollary}


\begin{theorem}\label{Theorem-3}
Let $R$ be a 2-torsion free noncommutative prime ring with involution $'\ast'$ that admits a $CE-$Jordan $\ast$-derivation $\delta:R\to R$ such that $[\delta(x),x]\in Z(R)$ for all $x\in R.$ Then either $\delta=0$ or $R$ is an order in a central simple algebra of dimension at most $4$ over its center.
\end{theorem}
\begin{proof}
Let us assume that $[\delta(x),x]\in Z(R)$ for all $x\in R.$ By Lemma \ref{Fact-1} and Lemma \ref{lemma-2}, we conclude that
\begin{equation}\label{eq-38}
[\delta(x),x]=0\mbox{~for~all~}x\in R.
\end{equation}
Polarizing (\ref{eq-38}), we have
\begin{equation}\label{eq-39}
[\delta(x),y]+[\delta(y),x]=0\mbox{~for~all~}x,y\in R.
\end{equation}
Replacing $y$ by $k^{2},$ where $k\in S(R)$ in (\ref{eq-39}), we arrive at
\begin{equation}\label{eq-40}
[\delta(x), k^{2}]+[-\delta(k)k+k\delta(k),x]=0\mbox{~for~all~}x\in R, \;k \in S(R).
\end{equation}
Since $\delta$ is a commuting and additive function thereby using Lemma \ref{L-4}, there exists $\lambda \in C$ and an additive mapping $\sigma:R\to C$ such that
\begin{equation}\label{eq-40.1}
\delta(x)=\lambda x +\sigma(x)\mbox{~for~all~}x\in R.
\end{equation}
With this, it follows from (\ref{eq-40}) that
\[
[\lambda x,k^{2}]+[-(\lambda k+\sigma(k))k+k(\lambda k+\sigma(k)),x]=0\mbox{~for~all~}x\in R, \;k \in S(R).
\]
In fact, we have
\begin{equation}\label{eq-41}
\lambda[ x, k^{2}]=0\mbox{~for~all~}x\in R, \;k \in S(R).
\end{equation}
Since the center of a prime ring contains no proper zero divisors, it follows that either $\lambda=0$ or $k^{2}\subseteq Z(R)\mbox{~for~all~}k \in S(R)$. In the latter case, we have $0=[k^{2},x]=[k,x]k+k[k,x]$ for all $x\in R$ and $k\in S(R).$ It can be seen as $d_{x}(k)k+kd_{x}(k)=0$ for all $k \in S(R)$, where $d_{x}(k)=[k,x]$ is the inner derivation induced by $x\in R$. In view of Lemma \ref{L-9}, either $R$ satisfies $s_{4}$ or $d_{x}(t)=0$ for all $t\in R$. Clearly, $d_{x}(t)=[x,t]=0$ for all $x,t\in R$ leads a contradiction. Thus we conclude that $R$ is an order in a central simple algebra of dimension at most $4$ over its center.
\par Now, if $\lambda=0,$ then from (\ref{eq-40.1}), we have $\delta(x)=\sigma(x)$ for all $x\in R$. In view of $(D)$ and torsion condition on $R,$ it implies
\[
\sigma(x^{2})-\sigma(x)x^{\ast}-x\sigma(x)\in Z(R)\mbox{~for~all~}x\in R,
\]
and so
\[
\sigma(x)[x+x^{\ast},y]=0\mbox{~for~all~}x,y\in R.
\]
Using Brauer's trick, we have either $\sigma(x)=0$ for all $x\in R$ or $x+x^{\ast} \in Z(R)$ for all $x\in R$. If $\sigma=0,$ then by (\ref{eq-40.1}) $\delta=0,$ as desired. Now if $x+x^{\ast} \in Z(R)$ for all $x \in R,$ then replacing $x$ by $h\in H(R),$ we get $H(R) \subseteq Z(R).$ Hence by Lemma \ref{L-8}, we are done.
\end{proof}

\begin{theorem}\label{Theorem-4}
Let $R$ be a $2$-torsion free noncommutative prime ring with involution $'\ast'$ that admits a $CE-$Jordan $\ast$-derivation $\delta:R\to R$ such that $[\delta(x),x^{\ast}]\in Z(R)$ for all $x\in R.$ Then either $\delta=0$ or $R$ is an order in a central simple algebra of dimension at most $4$ over its center.
\end{theorem}
\begin{proof}
Let us consider $[\delta(x),x^{\ast}]\in Z(R)$ for all $x\in R.$ In light of Lemma \ref{Fact-2} and Lemma \ref{lemma-2}, we may infer that
\begin{equation}\label{eq-42}
[\delta(x),x^{\ast}]=0\mbox{~for~all~}x \in R.
\end{equation}
This implies that
\begin{equation}\label{eq-43}
[\delta(x)^{\ast},x]=0\mbox{~for~all~}x \in R.
\end{equation}
Using Lemma \ref{L-4} in Eq. (\ref{eq-43}), there exists $\lambda \in C$ and an additive mapping $\sigma : R \to C$ such that
\[
\delta(x)^{\ast}=\lambda x+\sigma(x)\mbox{~for~all~}x \in R,
\]
and equivalently we have
\begin{equation}\label{eq-44}
\delta(x)=\lambda^{\ast}x^{\ast}+\sigma(x)^{\ast}\mbox{~for~all~}x \in R.
\end{equation}
Polarizing (\ref{eq-42}), we find
\[
[\delta(x),y^{\ast}]+[\delta(y),x^{\ast}]=0\mbox{~for~all~}x,y \in R.
\]
Replacing $y$ by $k^{2}$ in the last expression, where $k \in S(R),$ we get
\[
[\delta(x), k^{2}]+[-\delta(k)k+k\delta(k),x^{\ast}]=0.
\]
Using (\ref{eq-44}), it yields
\begin{equation}\label{eq-46}
[\lambda^{\ast} x^{\ast}, k^{2}]+[-(\lambda^{\ast} k^{\ast}+\sigma(k))k+k(\lambda^{\ast} k^{\ast}+\sigma(k)),x^{\ast}]=0\mbox{~for~all~}x \in R,~k \in S(R).
\end{equation}
Since $\lambda \in C$ and $\sigma(R)\subseteq C,$ we obtain
\[
\lambda^{\ast}[ x^{\ast}, k^{2}]=0\mbox{~for~all~}x \in R,~k \in S(R).
\]
Further, a similar demonstration that is given in the proof of Theorem \ref{Theorem-3} ensures the conclusion.
\end{proof}

\section*{Conflict of interest}
Here we declare that there is no conflict of interest in this work.


\end{document}